\documentclass{amsart}
\usepackage{amssymb}
\newtheorem{proposition}{Proposition}
\newtheorem{theorem}{Theorem}
\newtheorem{lemma}[theorem]{Lemma}

\theoremstyle{proof}
\theoremstyle{definition}

\theoremstyle{remark}

\numberwithin{equation}{section}

\newcommand{\abs}[1]{\lvert#1\rvert}

\begin{document}
\title{Squarefree integers and the abc conjecture}
\author{Zenon B. Batang}
\address{Coastal and Marine Resources Core Lab, King Abdullah University of Science and Technology, Thuwal 23955, Saudi Arabia}
\email{zenon.batang@kaust.edu.sa}

\date{July 31, 2021}

\begin{abstract}
For coprime positive integers $a, b, c$, where $a+b=c$, $\gcd(a,b,c)=1$ and $1\leq a < b$, the famous $abc$ conjecture (Masser and Oesterl\`e, 1985) states that for $\varepsilon > 0$, only finitely many $abc$ triples satisfy $c > R(abc)^{1+\varepsilon}$, where $R(n)$ denotes the radical of $n$. We examine the patterns in squarefree factors of binary additive partitions of positive integers to elucidate the claim of the conjecture. With $abc$ hit referring to any $(a, b, c)$ triple satisfying $R(abc)<c$, we show an algorithm to generate hits forming infinite sequences within sets of equivalence classes of positive integers. Integer patterns in such sequences of hits are heuristically consistent with the claim of the conjecture.
\end{abstract}

\maketitle

\section{Introduction}
Denote by $\mathbb{N}=\{1, 2, 3, \dots\}$ the set of natural numbers and $\mathbb{P}=\{2, 3, 5, \dots\}$ the set of prime numbers. An integer $n$ is \textit{squarefree} if it has no repeated prime factors, i.e. if $p | n$, then $p^{2} \nmid n$ for all prime $p$ dividing $n$. If $p^2 | n$, then $n$ is \textit{nonsquarefree} (also \textit{powerful}  or \textit{squareful}). Let $\mathbb{S}$ be the set of positive squarefree integers, including $1$ as the empty (nullary) product, with the first few elements as
\[
	\mathbb{S} = \{1, 2, 3, 5, 6, 7, 10, 11, 13, 14, 15, 17, 19, 21, 22, 23, \dots\}.
\]
Obviously, $\mathbb{P} \subset \mathbb{S} \subset \mathbb{N}$. Let $\omega(n)$ denote the number of distinct prime factors of $n$, where $\omega(1)=0$. Define
\begin{equation} \label{eq: 1.1}
	D_{\omega} = \{s \in \mathbb{S} : \omega = \omega(s)\},
\end{equation}
such that
\[
	D_0 = \{1\}, \quad D_1 = \mathbb{P}, \quad D_2 = \{6, 10, 14, \dots \}, \quad D_3 = \{30, 42, 66, \dots\}, \quad \dots
\]
Then we can also write
\begin{equation} \label{eq: 1.2}
	\mathbb{S} = \{D_{\omega}\}_{\omega = 0}^{\infty} = \{D_0, D_1, D_2, D_3, \dots\}.
\end{equation}
For $n \in \mathbb{N}$, define the finite set
\begin{equation} \label{eq: 1.3}
	\mathcal{S}(n) = \{s \in \mathbb{S} : s \leq n\}.
\end{equation}
For example, $\mathcal{S}(5) = \{1, 2, 3\}$ and $\mathcal{S}(10) = \{1, 2, 3, 5, 6, 7\}$. It is well known \cite{BM} that
\begin{equation} \label{eq: 1.4}
	 \abs{\mathcal{S}(n)} = \dfrac{6n} {\pi^{2}} + \mathcal{O}(n^{1/2}),
\end{equation}
where $\abs{A}$ denotes the \textit{cardinality} of $A$ and $\mathcal{O}$ is the usual \textit{big oh} notation. Put
\[
	\mathbb{N}_{>2} = \mathbb{N} \setminus \{1, 2\}, \qquad \mathbb{P}_{> 2} = \mathbb{P} \setminus \{2\}, \qquad \mathbb{S}_{>1} = \mathbb{S} \setminus \{1\}.
\]
Let $a, b, c$ be pairwise coprime positive integers such that
\begin{equation} \label{eq: 1.5}
	a + b = c,\quad \gcd(a,b,c)=1, \quad 1\leq a < b < c.
\end{equation}  
Obviously, all admissible values of $c$ are in $\mathbb{N}_{>2}$. Any $(a,b,c)$ tuple satisfying $\eqref{eq: 1.5}$ is herein termed \textit{abc triple}, where
\[
	a \equiv b \equiv 1 \pmod{2} \quad \text{if } \; c \; \text{ is even}, \qquad a \not \equiv b \pmod{2} \quad \text{if } \; c \; \text{ is odd}.
\] 
For $n >1$, define the \textit{radical}, also called \textit{squarefree kernel}, as
\begin{equation} \label{eq: 1.6}
	R(n) = \prod_{p \mid n} p,
\end{equation} 
where $R(1)=1$. Clearly, the radical is the highest squarefree factor of an integer. In 1985, Masser and Oesterl\`e \cite{JO, DWM} proposed the now famous \textit{abc conjecture}, which asserts that for $\varepsilon > 0$ only finitely many $abc$ triples satisfy
\begin{equation} \label{eq: 1.7}
	c > R(abc)^{1+\varepsilon}.
\end{equation} 
This simply stated but very deep Diophantine problem has far-reaching implications in number theory \cite{MW, MM}.

Here, we call any $abc$ triple satisfying $R(abc)<c$ an \textit{abc hit} \cite{MW}, e.g. $(1,2^3,3^2)$, $(5,3^3,2^5)$, $(2, 3^{10} \cdot 109, 23^{5})$ and $(11^2, 3^2 \cdot 5^6 \cdot 7^3, 2^{21} \cdot 23)$ are $abc$ hits. Most authors use \textit{abc triple} to mean \textit{abc hit}, but the distinction between terms should be clear in the present context. Write
\begin{equation} \label{eq: 1.8}
	c = Q(c) R(c), \quad Q(c) \in \mathbb{N},
\end{equation} 
where $Q(c)$ indicates how far $c$ is from being a squarefree or, equivalently, how rich $c$ in squareful factors \cite{NE}. With \eqref{eq: 1.8}, every $abc$ hits implies that
\begin{equation} \label{eq: 1.9}
	R(ab) < Q(c), \qquad R(abc) < c,
\end{equation}
where $R(ab) = R(a)R(b)$ by the multiplicative property of the radical. With \eqref{eq: 1.9}, one can quickly infer if a given $c$ is a \textit{no-hit number}, by which we mean that $c$ cannot be partitioned into $a$ and $b$ satisfying \eqref{eq: 1.5} and $R(abc) < c$. Obviously, every $c \in \mathbb{S}$ is a no-hit number since $Q(c)=1$ and, hence, $R(abc)>c$. As shown below, there also exist infinitely many nonsquarefree integers ($c \not \in \mathbb{S}$) that are no-hit numbers.

An $abc$ triple is termed \textit{primitive} if $(a,b,c)$ are pairwise coprime. For such triples where $c>2$, the conditions in \eqref{eq: 1.5} imply the trivial fact that
\begin{equation} \label{eq: 1.10}
	1 \leq a < \dfrac{c}{2} < b \leq c-1.
\end{equation}
Let $N(c)$ be the number of primitive triples for $c>2$. It is easy to see that
\begin{equation} \label{eq: 1.11}
	N(c) = \dfrac{\varphi(c)}{2}, \qquad \varphi(c) = c \prod_{p \mid c} \bigg(1-\dfrac{1}{p}\bigg),
\end{equation}
where $\varphi(c)$ is the \textit{Euler totient function}. Denote by $\mathcal{N}(c)$ the cumulative number of primitive triples up to $c$, with a known asymptotic formula \cite{FM} as
\begin{equation} \label{eq: 1.12}
	\mathcal{N}(c) = \sum_{k \leq c} N(k) = \dfrac{3c^2}{2 \pi^2} + \mathcal{O}(c \log c).
\end{equation}
There have been refinements of the error term in \eqref{eq: 1.12}, of which Walfisz \cite{AW} obtained the best unconditional estimate with $\mathcal{O}\big(c(\log c)^{2/3}(\log\log c)^{4/3}\big)$. It follows that
\begin{equation} \label{eq: 1.13}
	N(p) = \dfrac{p-1}{2}, \quad \forall \, p  \in \mathbb{P}_{>2}; \qquad N(p^{n}) = \dfrac{p^{n-1}(p-1)}{2}, \quad \forall \, n > 1.
\end{equation}
Both \eqref{eq: 1.12} and \eqref{eq: 1.13} imply the infinitude of primitive triples as $c$ tends to infinity. The $abc$ conjecture asserts that, among such triples, there are only rare occurrences of hits above the bound in \eqref{eq: 1.7} for any $\varepsilon > 0$. This suggests that, in most cases, a positive integer and its coprime binary partitions cannot be all composed of powers of small primes, i.e. when $c$ is composed of powers of small primes, then it is more likely to have coprime binary partitions $a$ and $b$ that are raised to powers of large primes or vice versa. The $abc$ conjecture states that there can only be finitely many exceptions to such a general pattern among the natural numbers, regardless of the magnitude of $\varepsilon > 0$.
 
Let $H(c)$ denote the number of hits for a given $c$. We say that $c$ is an \textit{exception} if $H(c)>0$. Note that the same $c$ may have multiple hits, of which the smallest are $81$ $\{(1, 2^4 \cdot 5, 3^4), (2^5, 7^2, 3^4)\}$ and $256$ $\{(13, 3^5, 2^8), (3^4, 5^2 \cdot 7, 2^8)\}$. Other examples are $3^8$ with $8$ hits and $5^6$ with $15$ hits. Where $H(c)>1$, we call $c$ a \textit{super exception}.

It follows from \eqref{eq: 1.11} and \eqref{eq: 1.13} that $N(p) > N(p+1)$ for $p>3$ and, as noted above, we have that
\begin{equation} \label{eq: 1.14}
		H(c) = 0, \quad \forall \, c \in \mathbb{S}.
\end{equation}
As shown below, there also exist infinitely many nonsquarefree $c$ where $H(c)=0$, although $\mathcal{N}(c)=\infty$ as $c$ goes to infinity. Let $\mathcal{H}(c)$ be the partial sum of the number of hits up to $c$, such that
\begin{equation} \label{eq: 1.15}
 	\mathcal{H}(c) = \sum_{k \leq c} H(k).
\end{equation}
Note that $\mathcal{H}(c)$ accounts for the multiplicity of hits due to the occurrences of super exceptions. Naturally, \eqref{eq: 1.15} diverges as $c$ goes to infinity, although we have yet to see an explicit formulation of $\mathcal{H}(c)$. With the infinitude of $\mathbb{S}$, we thus expect that
\begin{equation} \label{eq: 1.16}
	\lim_{c \to \infty} \dfrac{\mathcal{H}(c)}{c} = 0.
\end{equation}

It is relevant to state the elementary fact that
\begin{equation} \label{eq: 1.17}
	1 \leq R(a) < R(ab) < Q(c) \quad \Longrightarrow \quad 1 < R(b) < \dfrac{Q(c)}{R(a)},
\end{equation}
where $R(b)$ is at least $2$ or $3$ if $c$ is odd or even, respectively, such that $R(abc)$ is at least $6$ and $(1, 2^3, 3^2)$ is the least hit for all $c \in \mathbb{N}_{>2}$. Let $n = Q(c)$ in \eqref{eq: 1.3} to define
\begin{equation} \label{eq: 1.18}
	\mathcal{S}'(Q(c)) = \{s \in \mathbb{S}_{>1} : s < Q(c)\}.
\end{equation}
Denote the \textit{primorial} of the largest prime below $Q(c)$ by
\[
	P(Q(c)) = \prod_{p < Q(c)} p,
\]
i.e. excluding $Q(c)$ when it is prime, and let
\begin{equation} \label{eq: 1.19}
	G(c) = \gcd \big(P(Q(c)), R(c)\big).
\end{equation}
Then we can construct the set
\begin{equation} \label{eq: 1.20}
	\mathcal{G}(c) = \{s \in \mathbb{S}_{>1} : s | G(c), s < Q(c)\}.
\end{equation}
We also define
\begin{equation} \label{eq: 1.21}
	D_{max} = \max \{s \in \mathbb{S} : s < Q(c)\}.
\end{equation}
For brevity, we can change notation to
\begin{equation} \label{eq: 1.22}
	R(a) = d_1, \quad R(b) = d_2, \quad R(c) = d_3, \quad R(ab) = D, \quad R(abc) = R,
\end{equation}
where $R$ is always even since $d_1 d_2 \not \equiv d_3$ (mod 2) for all $c \in \mathbb{N}_{>2}$. Therefore, every hit must satisfy
\begin{equation} \label{eq: 1.23}
	d_1 d_2 = D \leq D_{max} < Q(c), \qquad d_1 d_2 d_3 = R < c.
\end{equation}
The term \textit{abc equation} shall refer to $a+b=c$, where any solution must satisfy the conditions that $d_1 d_2 d_3 < c$ and $\gcd(d_1, d_2, d_3) = 1$. Obviously, one finds that
\begin{equation} \label{eq: 1.24}
	\omega(d_1) \geq 0, \qquad \omega(d_2) \geq 1, \qquad \omega(D) \geq 1,
\end{equation}
where $D$ is squarefree as $d_1$ and $d_2$ are coprime and both squarefree. It follows that
\begin{equation} \label{eq: 1.25}
	\omega(D) = \omega(d_2) \quad \text{if } \; a = 1, \qquad \omega(D) = \omega(d_1) + \omega (d_2) \quad \text{if } \; a > 1.
\end{equation}
Further let $\mathcal{D}(c)$ be the set of admissible values of $D$ for any $c>2$ where $Q(c)>1$. By \textit{admissibility}, we mean that $1 < D \leq D_{max}$. Hence, $\mathcal{D}(c)$ is the \textit{search domain} for solutions to the $abc$ equation for a given $c$, i.e. $\mathcal{D}(c)$ constitutes the $(d_1, d_2)$ \textit{data} for the construction of $abc$ equations. It is easy to see that
\begin{equation} \label{eq: 1.26}
	\mathcal{D}(c) = \mathcal{S}'(Q(C)) \setminus \mathcal{G}(c),
\end{equation}
where $D \in \mathcal{D}(c)$ may not yield integral solutions to the $abc$ equation for $c$ or a class of $c$ with the same $Q(c)>1$. In fact, $\mathcal{D}(c)$ is not defined for all $c \in \mathbb{S}$, as claimed in \eqref{eq: 1.14}. We shall again invoke \eqref{eq: 1.26} in our argument below.

\section{(Not) as simple as abc}
The radical function is at the heart of the $abc$ conjecture. Let us call $D=R(ab)$ a \textit{product kernel} and $d_3=R(a+b)$ a \textit{sum kernel}, with both kernels being squarefree. As the radical is multiplicative, $R(abc)$ is a mix of both kernels since $R(abc) = Dd_3$, where $a+b=c$. Hence, the $abc$ conjecture encodes the additive and multiplicative structures in the ring of integers.

It is a fundamental fact that every positive integer can be uniquely represented as a product of prime powers, such that any two coprime positive integers will have distinct radicals and their sum will always yield a different radical, i.e. $d_1 \neq d_2 \neq d_3$. By comparing the product $d_1d_2d_3$ with the sum $c$, the $abc$ conjecture claims that only finitely many $abc$ triples satisfy $c>(d_1d_2d_3)^{1+\varepsilon}$ for any $\varepsilon > 0$. This motivates a deep inquiry into how the natural numbers are formed by the simple actions of addition and multiplication.

Let $f : \mathbb{S} \to \mathbb{N}$ maps $R(n)$ to $n$. No explicit formulation exists for such an \textit{inverse} operation. Apparently, $f(R(abc))$ is a one-to-many relation since there exist distinct triples $(a_i, b_i, c_i)$ and $(a_j, b_j, c_j)$ where $R(a_i b_i c_i) = R(a_j b_j c_j)$ for $c_i \neq c_j$, in which case $f(R(c_i)) = c_i$ and $f(R(c_j)) = c_j$. Suppose $A(c)$ and $B(c)$ denote the sets of $a$ and $b$ satisfying \eqref{eq: 1.5}, respectively, for $c>2$. With the bijection between $A(c)$ and $B(c)$, as implied by $\varphi(c)/2 = \varphi(c/2)$, we have that
\begin{equation} \label{eq: 2.1}
	\lvert A(c) \rvert = \lvert B(c) \rvert = N(c) > 0, \quad \forall \, c \in \mathbb{N}_{>2}.
\end{equation}
That is, among $a \in A(c)$ and $b \in B(c)$ that are coprime binary partitions of $c$, our aim is to find pairs $(a, b)$ satisfying \eqref{eq: 1.7}. Recall that $N(c)$ counts the pairs $(a, b)$ satisfying \eqref{eq: 1.5}, among which $H(c)$ is the number of pairs for which $d_1 d_2 d_3 < c$, with $N(c) > 0$ and $H(c) \geq 0$ for all $c>2$. It is implicit from \eqref{eq: 1.11} that $N(c)$ tends to be lower when $c$ consists of small prime factors and higher when $c$ is a large prime or a product of large prime powers. Hence, $N(c)$ tends to vary inversely with $\omega(c)$ and generally fluctuates around an increasing average trend as $c$ goes to infinity. A key question thus remains on how $H(c)$ behaves with increasing $c$.

For primitive triples, we have noted above that $a$ and $b$ are both odd if $c$ is even and of opposite parities if $c$ is odd. In order to yield hits, it follows that
\begin{equation} \label{eq: 2.2}
	3 \leq D < Q(c) \quad \text{if } \; c \; \text{ is even}, \qquad 2 \leq D < Q(c) \quad \text{if } \; c \text{ is odd}.
\end{equation}
Observe that $Q(c)$ expresses the relative magnitude of $c$ as an integer multiple of $d_3$. If there exists an $f$ such that $f(d) = n$ for $R(n)=d$, then the $abc$ conjecture claims that for any $c \in \mathbb{N}_{>2}$, it is rare to find integral solutions to
\begin{equation} \label{eq: 2.3}
	F(d_1, d_2) = f(d_1) + f(d_2) = c,
\end{equation}
with $d_1$ and $d_2$ satisfying \eqref{eq: 1.23}. The conjecture suggests that such solutions, if they exist, will mostly yield values of $R(abc)$ that are not much smaller than $c$. Currently, however, there is no effective machinery to deal with the radical inversion to solve \eqref{eq: 2.3}, but there are several techniques for finding $abc$ hits based on number theory or related fields (e.g. see review in \cite{MM}). Below, we outline an approach to solving equation \eqref{eq: 2.3} for specific integer classes.

Let $e_{p}(n)$ denote the \textit{p-adic order} of $n \in \mathbb{N}$, i.e. $e_{p}(n)$ is the highest exponent for which $p^{e} | n$. Then every $abc$ triple takes the form
\begin{equation} \label{eq: 2.4}
	a = \prod_{i \in \mathbb{N}} p_{i}^{e_{i}}, \qquad b = \prod_{j \in \mathbb{N}} p_{j}^{e_{j}}, \qquad c = \prod_{k \in \mathbb{N}} p_{k}^{e_{k}},
\end{equation}
where $e_{p}(n) = 0$ for $p \nmid n$. For an $abc$ hit, we have that
\begin{equation} \label{eq: 2.5}
	d_1 = \prod_{p_i | a} p_{i}, \qquad d_2 = \prod_{p_j | b} p_{j}, \qquad d_3 = \prod_{p_k | c} p_{k},
\end{equation}
where all $p_i, p_j, p_k$ are pairwise coprime. Hence, the difficulty of the $abc$ conjecture stems from the complexity of inverting \eqref{eq: 2.5} to yield \eqref{eq: 2.4} as integral solution to $abc$ equation. In addition to \eqref{eq: 2.3}, one may also find hits by solving the $abc$ equation in the form
\begin{equation} \label{eq: 2.6}
	E(d_3, d_2) = c - f(d_2) = a \quad \textnormal{or} \quad E(d_3, d_1) = c-  f(d_1) = b
\end{equation}
for any given $c \in \mathbb{N}_{>2}$. For obvious reason, it suffices to deal with integral solutions to $F(d_1, d_2) = c$ and $E(d_3, d_1) = b$ to find $abc$ hits.

\section{Finding abc hits}
The sum kernel $d_3 \in \mathbb{S}_{>1}$ is the largest squarefree divisor of $c \in \mathbb{N}_{>2}$. Let
\begin{equation} \label{eq: 3.1}
	R(Q(c)) = q,
\end{equation}
where $q \in \mathbb{N}$ and $Q(c)=q$ for all $Q(c) \in \mathbb{S}$. Also note that $q$ is the highest squarefree factor of $Q(c)$, such that $c$ is squarefree when $q=1$ and nonsquarefree when $q >1$. We define the \textit{equivalence class}
\begin{equation} \label{eq: 3.2}
	\overline{Q} = \{c \in \mathbb{N}_{>2} : Q(c) \textnormal{ is constant}\},
\end{equation}
where any distinct pairs $(c, d_3)$ and $(c', d'_3)$ exhibit an equivalence relation, denoted by $(c, d_3) \sim (c', d'_3)$, which satisfies $cd'_3 =c'd_3$. In such a case, $Q(c) = Q(c')$ such that $d_3$ and $d'_3$ are uniquely defined for $c$ and $c'$, respectively. Hence, this qualifies $\overline{Q}$ as an equivalence class of infinitely many $c \in \mathbb{N}_{>2}$, where $c = Q(c)d_{3}$ for a fixed $Q(c)$ and $d_3 \in \mathbb{S}_{>1}$. This means that all $c \in \overline{Q}$ have a constant $Q(c)$ and a varying $d_3$, e.g. all elements in $\{9, 18, 45, 63, 90, 99, \dots\}$, denoted by $\overline{Q}=\overline{3}$, are equivalent. Below, we show that the elements in $\overline{Q}$ may not share the same \textit{solvability} property in the context of the $abc$ conjecture, i.e. there may not exist an integral solution to $abc$ equation for some $c \in \overline{Q}$. In fact, some $\overline{Q}$ consist entirely of infinitely many no-hit numbers.

To find hits for $c \in \overline{Q}$, define $\mathcal{D}(c)$ from $Q(c)$ and then check for solutions to, say, $F(d_1, d_2) = c$, with $d_1$ and $d_2$ depending on $D \in \mathcal{D}(c)$. For the initial $c$, write
\begin{equation} \label{eq: 3.3}
	c = zK, \qquad \textnormal{where} \quad z = qQ(c), \quad d_3=qK, \quad \gcd(z, K) = 1.
\end{equation}
With \eqref{eq: 2.3}-\eqref{eq: 2.5}, the $abc$ equation takes the form
\begin{equation} \label{eq: 3.4}
	F(d_1, d_2) = f\big(\prod_{p_i | d_1} p_{i}\big) + f\big(\prod_{p_j | d_2} p_{j}\big) =  \prod_{p_i | d_1} p_{i}^{e_{p_i}} + \prod_{p_j | d_2} p_{j}^{e_{p_j}} = zK,
\end{equation}
which implies that
\begin{equation} \label{eq: 3.5}
	F(d_1, d_2) = \prod_{p_i | d_1} p_{i}^{e_{p_i}} + \prod_{p_j | d_2} p_{j}^{e_{p_j}} \equiv 0\pmod{z}.
\end{equation}
Notice that \eqref{eq: 3.5} may yield integral solutions for infinitely many $c$, not necessarily in the same equivalence class, as shown below. For finding solutions to $abc$ equations, consider the \textit{modular exponentiation}
\begin{equation} \label{eq: 3.6}
	p^{e} = r_{e} \pmod{m}, \quad e = 1, 2, 3, \dots, \quad \gcd(p, m)=1,
\end{equation}
such that for a fixed \textit{power vector} $e(n) = (e_i)_{i=1}^{\omega(n)}=(e_1, e_2, \dots, e_{\omega(n)})$, where $1 \leq i \leq \omega(n)$ indicates all $p_i | n$, we can write for any $n$
\begin{equation} \label{eq: 3.7}
	\prod_{i=1}^{\omega(n)} p_i^{e_i} \equiv \prod_{i=1}^{\omega(n)} r_{e_i} \pmod{m}.
\end{equation}
Hence, we can recast \eqref{eq: 3.5} into the form
\begin{equation} \label{eq: 3.8}
	F(d_1, d_2) \equiv \prod_{p_i | d_1} r_{e_i}  + \prod_{p_j | d_2} r_{e_j}  \equiv 0\pmod{z},
\end{equation}
which simplifies to
\[
	r_{e(a)} + r_{e(b)} \equiv 0\pmod{z}, 
\]
where
\[
	r_{e(a)} \equiv \prod_{p_i | d_1} r_{e_i} \pmod{z}, \qquad r_{e(b)} \equiv \prod_{p_j | d_2} r_{e_j} \pmod{z}.
\]
Let the solution set, if it exists, be denoted by
\begin{equation} \label{eq: 3.9}
	U_{z}(d_1; d_2) = \big\{\big(u_{z}(d_1); u_{z}(d_2)\big)\big\} = \big\{\big(e(a); e(b)\big)\big\},
\end{equation}
where
\begin{equation} \label{eq: 3.10}
	u_{z}(d_1) = e(a) = \big(e_i\big)_{i=1}^{\omega(d_1)}, \quad u_{z}(d_2) = e(b) =  \big(e_j\big)_{j=1}^{\omega(d_2)}.
\end{equation}
As specified, $e(a)$ and $e(b)$ also denote $u_{z}(d_1)$ and $u_{z}(d_2)$, respectively, for notational convenience. This suggests an increased complexity of the solution set with higher permutations of primes in $D$ and higher powers of primes in $(d_1, d_2)$ or, alternatively, with higher $Q(c)$ such that $G(c) \ll P(Q(c))$. Since \eqref{eq: 3.6} generates a \textit{cyclic} sequence, such a modular exponentiation presupposes that \eqref{eq: 3.8} may produce infinitely many solutions, which are hits for an infinitude of $c$ in different $\overline{Q}_i$.

As outlined above, solving the $abc$ equation can take a sequence of key steps as
\begin{equation} \label{eq: 3.11}
	c \in \mathbb{N}_{>2} \; \longrightarrow \; Q(c) \; \longrightarrow \; D \in \mathcal{D}(c) \; \longrightarrow \; F(d_1, d_2)=c, E(d_3, d_1) = b.
\end{equation}
That is, from an initial $c>2$ one obtains $Q(c)$ by \eqref{eq: 1.8} to identify all admissible product kernels $D \in \mathcal{D}(c)$ in accordance to \eqref{eq: 1.26}. Then for $(d_1, d_2)$ from any choice of $D$, one checks for integral solutions to $F(d_1, d_2) = c$ or $E(d_3, d_1)=b$ by modulo arithmetic. If a solution exists, then one naturally finds other solutions for infinitely many $c'>c$, where the distribution of $c'$ is further clarified below. 

Define
\begin{equation} \label{eq: 3.12}
	F_q = \{\overline{Q}_{i} : R(Q(c)) = v_{i}(q), \; \forall \; i \in \mathbb{N}, \; q \in \mathbb{S}_{>1}\},
\end{equation}
as the set of equivalent classes of $c \in \mathbb{N}_{>2}$ where $Q(c)$ is a function of $q$, denoted by $v_{i}(q)$, for all $\overline{Q}_i$. Specifically, $\overline{Q}_i$ is the equivalence class for $Q(c) = v_{i}(q)$, where $v_i (q) = q^i$ for $i \in \mathbb{N}$ if $q \in \mathbb{P}$. Write $\overline{Q}_i = \overline{q^i}$, such that
\[
	F_q = \{\overline{Q}_i\}_{i=1}^{\infty} = \{\overline{q^i}\}_{i=1}^{\infty} = \{\overline{q}, \overline{q^2}, \overline{q^3}, \dots\}, \quad \forall \; i \in \mathbb{N}.
\]
If $q \in \mathbb{S}_{>1}$ but $q \not \in \mathbb{P}$, write $\overline{Q}_i = \overline{v_i}$ where $\overline{v_i}$ corresponds to
\[
	v_{i} (q) = \prod_{j=1}^{n} p_{j}^{e_j}, \quad e_j \in \mathbb{N}, \quad \omega(q) = n > 1, \quad \forall \; i \in \mathbb{N},
\]
such that
\[
	F_q = \{\overline{Q}_i\}_{i=1}^{\infty} = \{\overline{v_i}\}_{i=1}^{\infty}, \quad v_{i}(q) < v_{i+1} (q),  \quad \forall \; i \in \mathbb{N}.
\]
For example, we have $F_3 = \big\{\overline{3}, \overline{9}, \overline{27}, \overline{81}, \overline{243}, \dots \big\}$ and $F_6 = \big\{\overline{6}, \overline{12}, \overline{18}, \overline{48}, \overline{54}, \dots \big\}$. 

For the least positive integer $t \in \mathbb{N}$ where
\[
	p^{t} \equiv 1 \pmod{m}, \quad t=1,2,3, \dots,
\]
with $1<p<m$ and $\gcd(p, m)=1$, then we call $t$ the \textit{period} of $p^e$ for which $r_e=1 \pmod{m}$ in \eqref{eq: 3.6}. If $k$ is the \textit{number of cycles} of $p^e$ with period $t$, then $r_{kt}=1 \pmod{m}$ fo all $k \in \mathbb{N}$. Let $T_{m}(p)$ be the cyclic sequence
\begin{equation} \label{eq: 3.13}
	T_{m}(p) = \{r_1, r_2, r_3, \dots\} \equiv \{p^e\}_{e=1}^{\infty} \pmod{m},
\end{equation}
such that for $e=\mu$ at $k=1$ for period $t$, we have that $r_{\mu} = r_{\mu+t(k-1)}  \pmod{m}$ for any $\mu, k, t \in \mathbb{N}$.

If $(a, b, c)$ is a solution to $F(d_1, d_2) = c$, where $c$ is \textit{minimal} in $F_{q}$, then we say that $(a, b, c)$ is a \textit{base hit} in $F_q$, e.g. $(1, 2^3, 3^2)$ is a base hit in $F_3$. From a base hit with $c \in F_q$, infinitely many other hits can be generated for $c' > c$, with such $c'$ not necessarily in the same $F_q$ since solving the $abc$ equation for $c \in F_q$ by modular arithmetic may yield solutions with $c' \not \in F_q$. This is an inherent outcome, as any modulo $m \in \mathbb{N}_{>1}$ may have a quotient with a squareful factor that is coprime to $m$, thus resulting in $q'$ for $c'$ where $q' \neq q$. For example, such cases arise in solutions to $F(d_1, d_2) \equiv 0 \pmod{z}$ for $c \in F_q$, with corresponding quotients, denoted by $K$ below, that instead yield $c'  \in F_{q'}$, where $q' \neq q$ and $\gcd(q', q)>1$ for $q, q' \in \mathbb{S}_{>1}$. Suppose that $q' > q$ and $\gcd(q', q)=q$, then there exists $s \in \mathbb{S}_{>1}$ such that $q'=sq$. In such a case, a relation between $F_{q'}$ and $F_q$ is specified below to characterize the distribution of $abc$ hits arising from a common base hit.

The above setup allows for inferences relevant to the claim of the $abc$ conjecture. We present some fundamental propositions pertaining to the distribution of $abc$ hits in $\mathbb{N}_{>2}$. Note that $H(c) > 0$ if $c$ has one or more hits and $H(c) = 0$ if $c$ is a no-hit number. We claim that

\begin{proposition} \label{prop1}
	$H(c)=0$ if $P(Q(c)) = G(c)$.
\end{proposition}

\begin{proof} [Proof]
	If  $P(Q(c)) = G(c)$, then $\mathcal{S}'(Q(c)) = \mathcal{G}(c)$ such that $\mathcal{D}(c) = \emptyset$ in \eqref{eq: 1.26} and, hence, $H(c)=0$ as claimed.
\end{proof}

The claim of the proposition is obvious, as no hit can occur since an $abc$ equation cannot be formed for $c$ with the indicated condition.

\begin{proposition} \label{prop2}
		$H(c) =  0$ for all $c \in \big\{\overline{Q}_{i}\big\} = \big\{\overline{1}, \overline{2}, \overline{4}, \overline{6}, \overline{8}, \overline{10}, \overline{12}\big\}$.
\end{proposition}

\begin{proof} [Proof]
Proposition \ref{prop1} holds for all $c$ in $\big\{\overline{Q}_{i}\big\}$. Hence, we only show that $H(c)=0$ if $P(Q(c)) \neq G(c)$ for each $\overline{Q}_{i}$. Obviously, $H(c)=0$ since $\mathcal{D}(c) = \emptyset$ for all $c \in \overline{1}, \overline{2}$. We have $\mathcal{D}(c)=\{3\}, \{5\}$ for all $c \in \overline{4}, \overline{6}$ but $H(c)=0$ for both classes since $1+3^e \not \equiv 0 \pmod{8}$ and $1+5^e \not \equiv 0 \pmod{36}$ for all $e \in \mathbb{N}$, respectively. For $\overline{8}$, we have at most $\mathcal{D}(c)=\{3, 5, 7\}$ but $H(c)=0$ since there are no integral solutions to $1+3^e=16K$, $1+5^e=16K$ and $1+7^e=16K$ for all $e \in \mathbb{N}$. Similarly, $H(c)=0$ for $\overline{10}, \overline{12}$, which have $\mathcal{D}(c)=\{3, 7\}, \{5, 7, 11\}$ but no solutions exist for $1+3^e = 100K, 1+7^e = 100K$ and $1+5^e = 72K, 1+7^e = 72K, 1+11^e = 72K$, respectively.
\end{proof}

Observe that if $Q(c)$ is even, then $d_3$ is even such that $c \equiv 0 \pmod{4}$. Also note that $\{1, 2, 4, 8\}$ are the only positive integers that cannot be written as a sum of a squarefree and a squareful number \cite{MP}. With $\overline{1}=\mathbb{S} \setminus \{1,2\}$, Proposition \ref{prop2} implies that the density of no-hit numbers far exceeds $6/\pi^2 (\approx 61\%)$. In fact, counting without multiplicity due to super exceptions, there are only $5$ and $22$ integers with hits up to $100$ and $1000$, respectively. With multiplicity, we have that $\mathcal{H}(100) = 6$ and $\mathcal{H}(1000)=31$.

\begin{proposition} \label{prop3}
	For $p, q \in \mathbb{P}$ and $e, e', \nu \in \mathbb{N}$, if $p^e \equiv r_e \pmod{q^{\nu}}$, then there exists $e'>e$ such that $p^{e'} \equiv r_{e} \pmod{q^{\nu}}$ and $p^{e'} \equiv r_{e'} \pmod{q^{\nu+1}}$ for $e' = e + {\eta}t$ and $r_{e'} = r_e + qt$, where $\eta \in \mathbb{N}$ and $t$ is the period of $T_{q^{\nu}}(p)$.
\end{proposition}

\begin{proof} [Proof]
	For $q$ a prime, it is well known that $t = \varphi(q^{\nu}) = q^{\nu-1} (q - 1)$, which easily implies that $t'=qt$ where $t'$ is the period of $T_{q^{\nu+1}}(p)$. Given $p^e \equiv r_e \pmod{q^{\nu}}$, which generates a cyclic sequence of $r_e$, there exists some integer $\eta>0$ such that $p^{e'} \equiv r_e \pmod{q^{\nu}}$ where $e' = e + {\eta}t$. It follows that $p^{e'} \equiv r_{e'} \pmod{q^{\nu+1}}$ where $r_{e'} = r_e + qt$, as claimed.
\end{proof}

For example, Proposition \ref{prop3} allows for the recursion of $2^3 = 8 \pmod{3^2}$ to yield $2^9 \equiv 26 \pmod{3^3}$, where $e'= 3 + (1)(6)$ and $r_{e'} = 8 + (3)(6)$ given $e=3, r_{e}=8, q=3, \eta = 1$ and $t=6$. In turn, with $t'=18$ for $2^9 \equiv 26 \pmod{3^3}$, the next recursion yields $2^{27} \equiv 80 \pmod{3^4}$, where $e' = 9 + (1)(18)$ and $r_{e'} = 26 + (3)(18)$. Such an algorithm runs recursively to generate infinitely many distinct congruence relations, thus providing a way to obtain hits from solutions to congruence relations in the form of \eqref{eq: 3.8}. By the multiplicative property of modulo arithmetic, Proposition \ref{prop3} can be extended to such cases as $\prod_{i=1}^{k} p_i^{e_i} \equiv r_{e} \pmod{q^{\nu}}$. Notice that Proposition \ref{prop3} is familiarly related to the well known \textit{Hensel's lemma}.

For $F_{q}$ and $F_{q'}$, where $q' = sq$ for  $s \in \mathbb{S}$ and $q, q' \in \mathbb{S}_{>1}$, define $F'_{q} \subseteq F_{q}$ as
\[
	F'_{q} = \{c \in F_{q} : s | d_3 \; \forall \;  c \in F_q, \; s \in \mathbb{S}\},
\]
where it is implicit that $\gcd(q, q')=q$ and $\gcd(s, q)=1$ such that $F_{q'} = F_{q}$ when $s=1$. Alternatively, $F_{sq}$ also denotes $F_{q'}$. We convene that for an integer $n$ and a set $A = \{x_1, x_2, x_3, \dots\}$, it holds that $nA = \{nx_1, nx_2, nx_3, \dots\}$. With this setup, we claim that

\begin{proposition} \label{prop4}
	$F_{sq} = F_{q'} \supseteq sF'_{q}$.
\end{proposition}

\begin{proof} [Proof]
	Since $s q^2 | c$ for all $c \in F'_{q}$ and $s^2q^2 | c'$ for all $c' \in F_{q'}$, then the claim of the proposition follows for all $s \in \mathbb{S}$, $q, q' \in \mathbb{S}_{>1}$, where $q'=sq$ and $\gcd(s,q)=1$.
\end{proof}

In subsequent claims, we will use $F_{sq}$ to denote $F_{q'}$, where it is understood that $s \not \equiv 0 \pmod{q}$ for $s, q, q'$ as specified above. Proposition \ref{prop4} establishes consistency in our argument, as it allows for assigning any $c$ generated from an $abc$ equation to a superset such that a claim can remain valid regardless of whether $c$ is finite or infinite in $F_{sq}$ for a fixed $s \in \mathbb{S}$. That is, even if $c$ is entirely in $F_q$, the statement that $c \in F_{sq}$ naturally holds since it is implicit that $F_{sq} = F_q$ for $s=1$. In fact, $F_{sq}$ is a convenient shorthand notation for
\[
	F_{sq} = \bigcup_{\substack{s \in \mathbb{S} \\ (s, q)=1}} sF'_{q},
\]
which is supported on all $s \not \equiv 0 \pmod{q}$ for a fixed $q \in \mathbb{S}_{>1}$.

\begin{proposition} \label{prop5}
	If there exists a solution to $F(d_1, d_2)=c$ for the least $c \in F_{q}$, then there may also exist solutions for infinitely many other integers greater than $c$ with the same $(d_1, d_2)$ in $F_{sq}$.
\end{proposition}

\begin{proof} [Proof]
	If there is an initial solution to $F(d_1, d_2)=c$ for the least $c \in F_q$, where such a solution results from evaluating $F(d_1, d_2) \equiv 0 \pmod{z}$ with $(d_1, d_2)$ as defined by the choice of $D \in \mathcal{D}(c)$, then such a solution must satisfy $d_1 d_2 d_3 < c$, with all notations as above. By Proposition \ref{prop3} and its obvious application, there also exist solutions to $F(d_1, d_2)=c'$ for $c'>c$ with the same $(d_1, d_2)$ and satisfying the same condition that $d_1 d_2 d_3 < c$, with $c' \equiv 0 \pmod{z}$. Since $c = zK$ and $c' = zK'$, where $K$, but not $K'$, is necessarily squarefree, then it holds that $c, c' \in F_{sq} = sF'_q$, where $s \in \mathbb{S}$, for all solutions by Proposition \ref{prop4}.
\end{proof}

Proposition \ref{prop5} accounts for the fact that if there is any solution to the $abc$ equation for $c \in F_q$  in the first cycles of $T_{z}(p)$ for all prime factors of $(d_1, d_2)$, then there also exist infinitely many solutions for $c' > c$ in subsequent cycles for the same $(d_1, d_2)$, where it generally holds that $c, c' \in F_{sq}$ for all such solutions.

\begin{proposition} \label{prop6}
	$H(c) = 0$ for all even $c \in \overline{Q}=\overline{3}$ but from the base hit with $c \in F_3$ we have that $H(c)>0$ for infinitely many odd $c \in F_{sq}$, where all such hits are in the form $1 + 2^e = 9K$ for $e = 3 + 6(k_2-1), k \in \mathbb{N}$ and $F_{sq} = sF'_3$ for $s \in \mathbb{S}$.
\end{proposition}

\begin{proof} [Proof]
Proposition \ref{prop1} easily holds for even $c \in \overline{3}$. For odd $c$, we have that $\mathcal{D}(c)=\{2\}$, such that $\mu_2=3$ and $t_2=6$ for $z=3^2$. Hence, the $abc$ equation takes the form $1 + 2^e = 9K$, with infinitely many solutions for $e = 3 + 6(k_2-1), k_2 \in \mathbb{N}$ and all such solutions have $c \in F_{sq}$ by Proposition \ref{prop5}.
\end{proof}

One can verify that the solutions to $1+2^e = 9K$ are all hits for $e = 3+6(k-1), k \in \mathbb{N}$. For example, we find the first six hits as
\[
\begin{aligned}
	& (1,2^3, 3^2), (1, 2^{9}, 3^{3} \cdot 19), (1, 2^{15}, 3^{2} \cdot 11 \cdot 331), (1, 2^{21}, 3^{2} \cdot 43 \cdot 5419),\\ 
	& (1, 2^{27}, 3^{4} \cdot 19 \cdot 87211), (1, 2^{33}, 3^{2} \cdot 67 \cdot 683 \cdot 20857),
\end{aligned}
\]
which all satisfy the condition that $d_1 d_2 d_3 < c$. Notice that $(1, 2^3, 3^2)$ is a base hit in $F_3$ and $H(c) = 0$ for $n \neq 3+6(k_2-1)$. One can verify that $c \in F_{sq}$, with $s=11$ at $k_2 = 28$ and $s = 19$ at $k_2=29$, thus consistent with the claim of Proposition \ref{prop6}.

\begin{proposition} \label{prop7}
	From the base hit with $c \in F_5$, there exist infinitely many hits of the forms $1 + 2^{e} = 25K$ for odd $c \in F_{sq}$ and $1 + 3^{e} = 25K$ for even $c \in F_{sq}$, where $e = 10 + 20(k-1)$ for $k \in \mathbb{N}$ in both cases and $F_{sq} = sF'_5$ for $s \in \mathbb{S}$.
\end{proposition}

\begin{proof} [Proof]
	Except when $6 | d_3$ where $H(c) = 0$ by Proposition \ref{prop1}, it is easy to establish by the same procedure as above that $\mathcal{D}(c) = \{2, 3\}$ for $c \in \overline{5}$, for which we find that $\mu_2=\mu_3=10$ and $t_2=t_3=20$ for $y=5^2$, such that there exist infinitely many integral solutions to $1 + 2^{e} = 25K$ and $1 + 3^{e} = 25K$, where $e=10 + 20(k-1)$ for $k \in \mathbb{N}$ regardless of the parity of $c$ and all hits from both equations have $c \in F_{sq}$. 
\end{proof}

The $abc$ hits arising from Proposition \ref{prop5} are exemplified by
\[
\begin{aligned}
	& (1, 2^{10}, 5^{2} \cdot 41), (1, 2^{30}, 5^{2} \cdot 13 \cdot 41 \cdot 61 \cdot 1321),\\
	& (1, 2^{50}, 5^3 \cdot 41 \cdot 101 \cdot 8101 \cdot 268501) \; \text{for odd} \; c, \\
	& (1, 3^{10}, 2 \cdot 5^{2} \cdot 1181), (1, 3^{30}, 2 \cdot 5^{2} \cdot 73 \cdot 1181 \cdot 47763361), \\
	& (1, 3^{50}, 2 \cdot 5^3 \cdot 101 \cdot 1181 \cdot 394201 \cdot 61070817601) \; \text{for even} \; c.
\end{aligned}
\]
So far, we have only dealt with infinitude of hits for $c$ whose product kernels consist of single primes, i.e. $\omega(D)=1$ where $a=1$. To illustrate the occurrences of hits where $\omega(D) > 1$, we consider $c \in \overline{Q}=\overline{7}$, which includes the least $D \in \mathcal{D}(c)$ where $\omega(D)=2$. As above, our task is to solve $F(d_1, d_2) \equiv 0 \pmod{z}$, where all solutions must satisfy $d_1 d_2 d_3 < c$, with $c=zK$ for a fixed $z$ and $K$ varies with the values and infinite exponentiations of $(d_1, d_2)$ in $F(d_1, d_2) = c$, such that all resulting hits have $c \in F_{sq}$.

\begin{proposition} \label{prop8}
	From the base hit with $c \in F_7$, it holds that $H(c)=0$ for $F(1,2) = c$ but $H(c)>0$ for $F(2,3)=c, F(3,2) = c$ and $F(1, 6) = c$ for odd $c \in F_{sq}$, whereas $H(c)>0$ for $F(1,3) = c$ and $F(1, 5) = c$ for even $c \in F_{sq}$, where $c=49K$ in all equations and $F_{sq} = sF'_7$ for $s \in \mathbb{S}$.
\end{proposition}

\begin{proof} [Proof]
	By Proposition \ref{prop1}, $H(c)=0$ if $d_3 \equiv 0 \pmod{30}$ for $c \in \overline{Q} = \overline{7}$, which has at most $\mathcal{D}(c)=\{2, 3, 5, 6\}$. Denote by $e_p$ the exponent of $p$ and $t_p$ the period of $T_z(p)$. For odd $c$, $H(c)=0$ for $F(1, 2)=c$ since $1 + 2^{e_2} \not \equiv 0 \pmod{49}$ for all $e_2 \in \mathbb{N}$. With periods $t_2=21$ and $t_3=42$ for $T_{49}(2)$ and $T_{49}(3)$, the solution sets $U_{49}(2; 3)$ and $U_{49}(3; 2)$ for $F(2,3)=49K$ and $F(3,2)=49K$ are formed by exponent pairs $(e_2, e_3)$ and $(e_3, e_2)$ that satisfy $2^{e_2}<3^{e_3}$ and $3^{e_3}<2^{e_2}$, respectively. The pairs $(e_2, e_3)$ and $(e_3, e_2)$ correspond to $r_{e_2}+r_{e_3} \equiv 0 \pmod{49}$ and $r_{e_3}+r_{e_2} \equiv 0 \pmod{49}$ for $r_{e_2} \in T_{49}(p)$ and $r_{e_3} \in T_{49}(p)$. Since $3$ is a \textit{primitive root} modulo $49$ but $2$ is not, every $r_{e_2}$ will pair with corresponding $r_{e_3}$ to form $U_{49}(2; 3)$ and $U_{49}(3; 2)$, which are both infinite as the conditions that $2^{e_2}<3^{e_3}$ and $3^{e_3}<2^{e_2}$ are always satisfied at higher cycles of $T_{49}(2)$ and $T_{49}(3)$.
	
	For the solution set $U_{49}(1, 6)$ for $F(1,6)=49K$, where $u_{49}(1) = 0$ since $a=1$, we only need to consider the solutions $u_{49}(6) = (e_2, e_3)$ to $r_{e_2}r_{e_3} \equiv -1 \pmod{49}$. One can check that $(1, 2^4\cdot3, 7^2)$ is the base hit in $F_7$. As both $e_2$ and $e_3$ tend to infinity, infinitely many pairs $(e_2, e_3)$ satisfy $2^{e_2}3^{e_3} \equiv -1 \pmod{49}$, thus establishing the infinitude of hits for odd $c$ in  $F_7$. Given $t_2$ and $t_3$ above, the solutions to $1+2^{e_2}3^{e_3}=49K$ arise from $(e_2, e_3) = \big(\mu_2 + t_{2}(k_2-1), \mu_3+ t_{3}(k_3-1)\big)$ for any $k_2, k_3 \in \mathbb{N}$ denoting the numbers of cycles in $T_{49}(2)$ and $T_{49}(3)$ while $\mu_2$ and $\mu_3$ are the values of $e_2$ and $e_3$ for which $r_{e_2}r_{e_3} \equiv -1 \pmod{49}$ at $k_2=k_3=1$, not necessarily for the base hit, respectively. Hence, $u_{49}(6)$ holds for $(r_{e_2}, r_{e_3}) = (1, 48), (2, 24), (4, 12), (16, 3)$ which correspond to $(\mu_2, \mu_3) = (21, 21), (1, 37), (2, 11), (4, 1)$, respectively.
	
	For even $c$, with $t_3=42$ as above and $t_5=21$, an infinitude of hits result from $1+3^{e} \equiv 1 + 5^{e} \equiv 0\pmod{49}$, where $e=21(2k-1)$ for $k \in \mathbb{N}$ in both congruences. In general, all hits arising from the above equations have $c \in F_{sq}$, where $F_{sq} = sF'_7$ for $s \in \mathbb{S}$, by Proposition \ref{prop5} and this completes the proof.
\end{proof}

To illustrate Proposition \ref{prop8}, some examples of hits for $F(2,3) = 49K$ are
\[
\begin{aligned}
	& (2, 3^{5}, 5 \cdot 7^2), (2^{3}, 3^{15}, 5 \cdot 7^{2} \cdot 58567), (2^{6}, 3^{9}, 7^{2} \cdot 13 \cdot 31), \\
	& (2^{16}, 3^{17}, 7^{3} \cdot 163 \cdot 2311), (2^{13}, 3^{23}, 7^{3} \cdot 274469933),
\end{aligned}
\]
while those for $F(3,2) = 49K$ are
\[
\begin{aligned}
	& (3^{3}, 2^{9}, 7^{2} \cdot 11), (3^{5}, 2^{22}, 7^{4} \cdot 1747), (3^{7}, 2^{14}, 7^{2} \cdot 379), \\
	& (3^{9}, 2^{27}, 7^{2} \cdot 11 \cdot 271 \cdot 919), (3^{15}, 2^{24}, 7^{2} \cdot 19 \cdot 67 \cdot 499).
\end{aligned}
\]
Apparently, commuting the values for $(d_1, d_2)$ for the same $D \in \mathcal{D}(c)$ yields different sets of infinitely many hits. For $F(1,6)=c$, where $c$ is odd, recursing on $(e_2, e_3)$ from the base hit, i.e. setting $\mu_2=4$ and $\mu_3=1$ for $k_2=1,2,3$ but keeping $k_3=1$ unchanged, results in the hits
\[
	(1, 2^{4} \cdot 3, 7^{2}), (1, 2^{25} \cdot 3, 7^{3} \cdot 269 \cdot 1091), (1, 2^{46} \cdot 3, 7^{2} \cdot 4308290459857).
\]
For even $c \in F_7$, the first two hits from $1 + 3^{e} = 49K$ and $1 + 5^{e} = 49K$, where $e=21 + 42(k-1)$ for $k \in \mathbb{N}$ in both cases, are
\[
\begin{aligned}
	& (1, 3^{21}, 2^2 \cdot 7^2 \cdot 43 \cdot 547 \cdot 2269), (1, 3^{63}, 2^2 \cdot 7^2 \cdot 19 \cdot 37 \cdot 43 \cdot 127 \cdot  547 \cdot 883 \cdot \\
		& \qquad 2269 \cdot 2521 \cdot 550554229) \; \textnormal{for} \; F(1,3); \\
	& (1, 5^{21}, 2 \cdot 3^2 \cdot 7^2 \cdot 29 \cdot 43 \cdot 127 \cdot 449 \cdot 7603), (1, 5^{63}, 2 \cdot 3^3 \cdot 7^2 \cdot 29 \cdot 43 \cdot 127 \cdot  \\
		&   \qquad 449 \cdot 883 \cdot 5167 \cdot 7603 \cdot 406729 \cdot 24132781 \cdot 1692416503) \; \textnormal{for} \; F(1,5).
\end{aligned}
\]
As claimed, notice that $c \in 2F'_7$ and $c \in 3F'_7$ in the first two hits for $F(1,3)=49K$ and $F(1,5)=49K$, respectively. One can verify that $F(3,2) = 49K$ also includes the hit  $(3, 2^{17}, 5^{2}\cdot 7^{2} \cdot 107)$, where $c \in 5F'_7$.

The above narrative so far dealt with hits for $c$ where $\omega(D) \leq 2$. Let us consider cases where $\omega(D) > \omega(d_3)$. Given $c \in \mathbb{N}_{>2}$, we can readily determine the complete search domain $\mathcal{D}(c)$, from which we can solve the $abc$ equation for $b$ in the form
\[
	c - f(d_1) = f(d_2), \qquad b = f(d_2), \qquad a = f(d_1),
\]
for any arbitrary choice of $d_1$ from $D \in \mathcal{D}(c)$, provided that $a < b$. As noted in \eqref{eq: 1.17}, every $d_2$ in a hit generally obeys
\begin{equation} \label{eq: 3.14}
	1 < d_2 < \dfrac{Q(c)}{d_1}.
\end{equation}
As above, we recast the $abc$ equation into a congruence relation as
\begin{equation} \label{eq: 3.15}
	r_{e(c)} - r_{e(a)} \equiv 0 \pmod{y}
\end{equation}
for a suitable choice of the modulus $y$, where we have that
\[
	r_{e(c)} = \prod_{p_k | d_3} r_{e_k} \pmod{y}, \qquad r_{e(a)} = \prod_{p_i | d_1} r_{e_i} \pmod{y}.
\]
Similarly, we denote the solution set by
\[
	U_{y}(d_3, d_1) = \big\{\big(u_{y}(d_3); u_{y}(d_1)\big)\big\} = \big\{\big(e(c); e(a)\big)\big\},
\]
where
\[
	u_{y}(d_3) = e(c) = \big(e_k\big)_{k=1}^{\omega(d_3)}, \qquad u_{y}(d_1) = e(a) = \big(e_i \big)_{i=1}^{\omega(d_1)},
\]
with $e(c)$ and $e(a)$ conveniently stand for $u_{y}(d_3)$ and $u_{y}(d_1)$, respectively. Without loss of generality, we illustrate this approach for $c$ a perfect prime power, where $\omega(d_3)=1$. Suppose that $c=q^{e(c)}$, such that $Q(c)=q^{e(c) - 1}$ for $q \in \mathbb{P}$ and $e(c) > 1$, from which we can enumerate all elements in $\mathcal{D}(c)$, then for any choice of $D \in \mathcal{D}(c)$, the $abc$ equation takes the form
\begin{equation} \label{eq: 3.16}
	E(d_3, d_1) = c - f(d_1) = f(d_2),
\end{equation}
where $c$ is given and $a = f(d_1)$ is set from $D$ to solve for $b=f(d_2)$ in the congruence
\begin{equation} \label{eq: 3.17}
	E(d_3, d_1) = c - f(d_1) \equiv 0 \pmod{y},
\end{equation}
with the solution set arising from all pairs $(r_{e(c)}, r_{e(a)})$ that satisfy the congruence relation, where $c > a$. From the given base hit, let $R(Q(b)) = w$ such that $d_2 = wk$, where $w > 1$ and $k \geq 1$ are both squarefree. It suffices to take $y = p^{e_p}$, where $p | w$ and $e_p>1$ is the least exponent such that $p^{e_p} | b$ and $p^{e_p} > d_1 d_3$. Note that $y$ is at least a square, such that $y \leq wQ(b)$ and $b = y\mathcal{K}$, where $\mathcal{K} \geq k$. With this setup, the $abc$ equation takes the form
\begin{equation} \label{eq: 3.18}
	f(d_1) + y\mathcal{K} = c, \quad \mathcal{K} \in \mathbb{N},
\end{equation}
where $\mathcal{K}$ is not necessarily squarefree and $R(y\mathcal{K})$ must satisfy \eqref{eq: 3.14} or $d_1 d_2 < Q(c)$. We thus claim that

\begin{proposition} \label{prop9} 
	For $c \in \mathbb{N}_{>2}$ of the form $c = q^{e(c)}$ where $q \in \mathbb{P}$, there exist infinitely many hits as integral solutions to $c - f(d_1) \equiv 0 \pmod{y}$, where $d_1$ and $y$ are defined from the base hit for $F(d_1, d_2) = c$ and all solutions $\big(e(c); e(a)\big)$ to $r_{e(c)} - r_{e(a)} \equiv 0 \pmod{y}$ must satisfy the conditions that $f(d_1) < q^{e(c)}$ and $d_1 d_2 < q^{e(c)-1}$.
\end{proposition}

\begin{proof} [Proof]
	As noted above, an $abc$ hit implies that $d_1 d_2 d_3 < c$. The  congruence relation in Proposition \ref{prop9} corresponds to an $abc$ equation of the form $q^{e(c)} - f( d_1) = y\mathcal{K}$, where $d_1$ and $y$ are specified from the base hit for $F(d_1, d_2) = c$, depending on the choice of $(d_1, d_2)$ from $D \in \mathcal{D}(c)$. Any pairs $\big(e(c); e(a)\big)$ satisfying $r_{e(c)}-r_{e(a)} \equiv 0 \pmod{y}$ is a solution to the $abc$ equation, provided that $f(d_1) < q^{e(c)}$ and $d_1 d_2 < q^{e(c)-1}$. With $d_2 = wk$ where $w>1$ and $k \geq 1$ are both squarefree and $d_3 = q$, it follows that any solution satisfying the specified conditions is always a hit since $w \leq d_2$ and, hence, $d_1 d_3 w \leq d_1 d_2 d_3 < c$. 
\end{proof}

To illustrate Proposition \ref{prop9}, suppose $c=2^{e_2}$ with the base hit $(5, 3^3, 2^5)$, such that $d_1=5$ and $y=3^3$ lead to $2^{e_2} - 5^{e_5} \equiv 0 \pmod{3^3}$. With $t_2 = t_5 = 18$ for $T_{27}(2)$ and $T_{27}(5)$, respectively, we find infinitely many hits from combinations of $(e_2 ; e_5)$ that satisfy the specified congruence, i.e all solutions $(e_2 ; e_5)$ to $r_{e(c)}-r_{e(a)} \equiv 0 \pmod{3^3}$ meeting the conditions that $2^{e_2} > 5^{e_5}$ and $d_1 d_2 < 2^{e_2-1}$. In fact, for any fixed $\mu_2$ and $\mu_5$, such solutions correspond to $(e_2 ; e_5) = \big(\mu_2 + t_{2}(k_2-1); \mu_{5} + t_{5}(k_5-1)\big)$, provided that both conditions are satisfied, where $\mu_2$ and $\mu_5$ are powers of $2$ and $5$ satisfying the congruence in the first cycles of $(r_{e(c)}, r_{e(a)})$ and $k_2$ and $k_5$ denote the numbers of cycles for $T_{27}(2)$ and $T_{27}(5)$, respectively. The hits arising from $E(2, 5) = b$, with $(5, 3^3, 2^5)$ as the base hit, include
\[
\begin {aligned}
	& (5^3, 3^4 \cdot 13 \cdot 31, 2^{15}), (5^8, 3^{5} \cdot 11 \cdot 1423, 2^{22}), \\
	& (5, 3^{6} \cdot 37 \cdot 311, 2^{23}), (5^6, 3^4 \cdot 7 \cdot 13 \cdot 31 \cdot 37 \cdot 127, 2^{30}).
\end{aligned}
\]
Despite conforming to the same congruence relation, note that the triples
\[
	(5^2, 3^3 \cdot 37, 2^{10}), (5^{4}, 3^3 \cdot 37 \cdot 1049, 2^{20}), (5^{8}, 3^3 \cdot 37 \cdot 1049 \cdot 1049201, 2^{40})
\]
are not hits as they fail the condition that $d_1 d_2 < Q(c)$. 

There are also infinitely many hits for $c = 2^{e_2}$ of the form $E(2, 3)=b$. With the base hit $(3, 5^3, 2^7)$, we have that $d_1=3$ and $y=5^2$ such that $t_2 = t_3 = 20$ for $T_{25}(2)$ and $T_{25}(3)$, from which we obtain infinitely many hits, such as
\[
\begin {aligned}
	& (3^4, 5^2 \cdot 7, 2^8), (3^2, 5^3 \cdot 131, 2^{14}), (3^3, 5^3 \cdot 19 \cdot 883, 2^{21}), \\
	& (3^4, 5^3 \cdot 13^2 \cdot 97 \cdot 131, 2^{28}), (3^5, 5^4 \cdot 54975581, 2^{35}).
\end{aligned}
\]
For comparison, let us also consider $c=3^{e_3}$, with $(1, 2^3, 3^2)$ as the base hit. With $t_1=1$ and $t_3=2$ for $y=4$, we find infinitely many hits, including
\[
\begin{aligned}
	& (1, 2^4 \cdot 5, 3^4), (1, 2^3 \cdot 7 \cdot 13, 3^6), (1, 2^5 \cdot 5 \cdot 41, 3^8), \\
	& (1, 2^3 \cdot 11^2 \cdot 61, 3^{10}), (1, 2^4 \cdot 5 \cdot 7 \cdot 13 \cdot 73, 3^{12}).
\end{aligned}
\]
Notice the lack of hits of the form $1 + 4\mathcal{K} = 3^{e_3}$ if $e_3$ is odd, thus suggesting that any hits in such a class have a different exponential form, e.g. $1+121\mathcal{K} = 3^{e_3}$ for $e_3=5(2k_3-1)$ which yields the first three hits 
\[
	(1, 2 \cdot 11^2, 3^{5}), (1, 2 \cdot 11^2 \cdot 13 \cdot 4561, 3^{15}), (1, 2 \cdot 11^2 \cdot 8951 \cdot 391151, 3^{25}).
\]
This clarifies why $3^2$ and $3^4$ are exceptional, but $3^3$ is not. In addition to $F(1, 10)=c$, we also find that $E(3,2)=b$ yields infinitely many hits for $c \in F_3$ by recursion from the base hit $(2^5, 7^2, 3^4)$. As $y$ is at least a square, let $y=7^2$ for which $t_2 = 21$ and $t_3 = 42$ and all hits are of the form $2^{e_2} + 49\mathcal{K} = 3^{e_3}$. The associated congruence has solutions $(e_3; e_2)$ in the form $e_3=\mu_3+42(k_3-1)$ and $e_2=\mu_2 + 21(k_2 - 1)$, where $\mu_p$ and $k_p$ are as above. In addition to the base hit, other hits resulting from $E(3,2)=b$ include
\[
\begin{aligned}
	& (2^2, 5 \cdot 7^2 \cdot 241, 3^{10}), (2^{10}, 7^2 \cdot 113, 3^8), (2^4, 5 \cdot 7^2 \cdot 241 \cdot 59053, 3^{20}), \\
	& (2^{20}, 5 \cdot 7^2 \cdot 37 \cdot 41 \cdot 113, 3^{16}), (2^{12}, 7^2 \cdot 13 \cdot 23 \cdot 31 \cdot 853, 3^{18}).
\end{aligned}
\]

As for $K$, it is also possible that $\mathcal{K}$ includes a squareful factor in the solutions to $E(d_3, d_1) = b$, such as $(3^4, 5^3 \cdot 13^2 \cdot 97 \cdot 131, 2^{28})$ for $E(2,3) = b$ and $(1, 2^3 \cdot 11^2 \cdot 61, 3^{10})$ for $E(3,1) = b$ above. Hence, while our approach can generate infinitely many hits with $F(d_1, d_2)=c$ and $E(d_3, d_2)=b$, it requires ascertaining if $K$ and $\mathcal{K}$ are indeed squarefree in the resulting $c$ and $b$, respectively.

The above narrative establishes the infinitudes of integers with or without hits, but remains insufficient to validate the claim of the $abc$ conjecture. Our next task is to provide further insights into underlying integer patterns that are fundamental to a better understanding of the inherent complexity of the conjecture.

\section{Anatomy of a conjecture}
Although there exist infinitely many hits, the $abc$ conjecture asserts that $d_1 d_2 d_3$ cannot be far smaller than $c$, as implied by the constraint imposed by the exponent $1 + \varepsilon$ in \eqref{eq: 1.7}. One of the key measures in characterizing $abc$ hits is termed \textit{quality}, denoted by $\lambda(a, b, c)$ and expressed as
\begin{equation} \label{eq: 4.1}
 	\lambda(a, b, c) = \dfrac{\log c}{\log (d_ 1 d_2 d_3)}.
\end{equation}
The $abc$ conjecture equivalently states that only finitely many triples satisfy
\begin{equation} \label{eq: 4.2}
	\lambda(a, b, c) > 1 + \varepsilon, \quad \forall \; \varepsilon > 0.
\end{equation}
Among the known hits so far, $(2, 3^{10} \cdot 109, 23^{5})$ has the highest quality with $\lambda(a, b, c) = 1.62991$, which was obtained by E. Reyssat in 1987, reportedly by brute force \cite{LZ}, although such a hit could be established by continued fraction approximation \cite{MM}. It is unknown if there are hits satisfying $\lambda(a, b, c) \geq 2$, for which a simpler proof of the \textit{Fermat's last theorem} would immediately follow. It is believed \cite{JB}, \cite{BB} that
\begin{equation} \label{eq: 4.3}
	\limsup_{c \to \infty} \lambda(a, b, c) = 1.
\end{equation}
If proven true, then \eqref{eq: 4.3} consequently affirms the claim of the $abc$ conjecture.

A proof of the $abc$ conjecture requires ascertaining that despite the infinitude of hits, only finitely many integers satisy \eqref{eq: 1.7} for $\varepsilon > 0$. For heuristic argument, let
\[
	\lambda(a, b, c) = 1 + \varepsilon (a, b, c), \qquad \theta (c)= \dfrac{Q(c)}{D}, \qquad D = d_1 d_2.
\]
It can be shown from \eqref{eq: 4.1} that
\begin{equation} \label{eq: 4.4}
	\varepsilon(a, b, c) = \dfrac{\log \theta (c)}{\log D + \log d_3},
\end{equation}
where $\theta (c) > 1$ expresses the magnitude of $Q(c)$ relative to $D$ for which $F(d_1, d_2) = c$ generates infinitely many hits, as demonstrated above. Such infinitude of hits with the same $(d_1, d_2)$ belong entirely to a specific $F_{sq}$. Suppose $(a, b, c)$ and $(a', b', c')$ are distinct hits with common $(d_1, d_2)$ in the same $F_{sq}$, where $c<c'$ and $a \leq a'$, then with $c = qQ(c)K$ and $c'=q'Q(c')K'$ it is almost surely true that
\begin{equation} \label{eq: 4.5}
	\dfrac{b'}{b} > \dfrac{Q(c')}{Q(c)} \; \Longrightarrow \; K < K' \; \Longrightarrow \;  d_ 3 < d'_3,
\end{equation}
which implies that $d_3=qK$ and $d'_3 =q'K'$, where $q \leq q'$. Propositions \ref{prop3}-\ref{prop5} provide the basis for establishing \eqref{eq: 4.5}, which tends to yield $\varepsilon(a', b', c') < \varepsilon(a, b, c)$ infinitely often as implied by the term ``almost surely". Since $(d_1, d_2)$ are fixed and $d'_3 > d_3$ for almost all $c' > c$ in $F_{sq}$, we thus expect for any fixed $\varepsilon > 0$ that
\begin{equation} \label{eq: 4.6}
	\varepsilon (a, b, c) \leq \varepsilon, \quad \forall \; c \in F_{sq}, \; s \in \mathbb{S},
\end{equation}
infinitely often. The same claim holds for hits from \eqref{eq: 3.18}, whence we have that
\begin{equation} \label{eq: 4.7}
	\dfrac{c'}{c} > \dfrac{Q(b')}{Q(b)} \; \Longrightarrow \; \mathcal{K} < \mathcal{K}' \; \Longrightarrow \;  d_ 2 < d'_2,
\end{equation}
where $b=wQ(b)\mathcal{K}$ and $b' = w'Q(b')\mathcal{K}'$, with $w \leq w'$. It also follows that $b < b'$, $a \leq a'$, $d_2=s\mathcal{K}$ and $d'_2=w'\mathcal{K}'$. Observe that $d'_3=d_3$ for all $c' > c$ but $d'_2 > d_2$ for almost all $b' > b$, such that $\varepsilon(a', b', c') < \varepsilon(a, b, c)$ infinitely often and \eqref{eq: 4.6} holds accordingly, with only finitely many exceptions for any fixed $\varepsilon > 0$. The same claim applies to any infinite set of hits in $F_{sq}$ for $q \in \mathbb{S}_{>1}$ and $s \in \mathbb{S}$, where $\gcd(s, q)=1$. 

Recall that for every base hit with $c \in F_q$, infinitely many hits for $c'>c$, with $c' \in F_{sq}$ for $s \in \mathbb{S}$, exist as solutions to $F(d_1, d_2)=c$ and $E(d_3, d_1) = f(d_2)$ (also to $E(d_3, d_2) = f(d_1)$ as shown below), where such hits are easily determined so long as $t_{p_i}, t_{p_j}$ and $t_{p_k}$ are known for all $p_i | d_1$, $p_j | d_2$ and $p_k | d_3$, respectively. It needs more rigor to establish \eqref{eq: 4.5} and \eqref{eq: 4.7}, which both imply the quotient structures arising from lifted solutions to congruences modulo prime powers. It suffices to show that $K'$ and $\mathcal{K'}$ are mostly increasing squarefree factors in infinitely many $c' \in F_{sq}$. This thus provides a rough sketch of potential proof strategy for the $abc$ conjecture.

For further insight into the complexity of the $abc$ conjecture, note that $(d_1, d_2)$ denotes the \textit{ordered pair} as inputs to $F(d_1, d_2) = f(d_1)+f(d_2)=c$, where $a=f(d_1)$ and $b=f(d_2)$ must satisfy the condition that $1 \leq a < b$. Suppose $D=p_i p_j$, where $p_i \neq p_j$, it is clear that $(p_i, p_j) \neq (p_j, p_i)$ and $F(p_i, p_j) \neq F(p_j, p_i)$. The additive commutativity of $abc$ equations is preserved, with $f(p_i) + f(p_j) = f(p_j) + f(p_1)$, provided that $a = f(p_i) < b = f(p_j)$.

Let $\mathcal{Q}_0$ be the set of equivalent classes of $c$ that are absolutely no-hit numbers by Proposition \ref{prop2}, i.e. put
\begin{equation} \label{eq: 4.8}
	\mathcal{Q}_0 = \big\{\overline{Q}_i\big\} = \big\{\overline{1}, \overline{2}, \overline{4}, \overline{6}, \overline{8}, \overline{10}, \overline{12} \big\}.
\end{equation}
A \textit{brute force} search for hits requires checking at most $N(c)$ pairs of $(a, b)$ for any $c \not \in \mathcal{Q}_0$. Our approach initially sets $F(d_1, d_2)=c$ or $E(d_3, d_1)=b$ for the base hit with a given $D \in \mathcal{D}(c)$, from which infinitely many solutions will follow naturally. For $\mathcal{D}(c)$ as defined in \eqref{eq: 1.26}, its cardinality is denoted by
\begin{equation} \label{eq: 4.9}
	\abs{\mathcal{D}(c)} = \# \{(d_1, d_2) \in \mathbb{S} \times \mathbb{S}_{>1} : d_1 d_2 = D \in \mathcal{D}(c) \}.
\end{equation}
Note that $\abs{\mathcal{D}(c)}$ may be constrained by Proposition \ref{prop1} for which $\abs{\mathcal{D}(c)}=0$ or by $G(c)$ in \eqref{eq: 1.19} for which $0 \leq \abs{\mathcal{D}(c)} \leq \abs{\mathcal{S}'(Q(c))}$ where $\mathcal{S}'(Q(c))$ is as defined in \eqref{eq: 1.18}. Let $M(c)$ denote the \textit{maximal} number of $abc$ equations, regardless if solvable or not, that can be formed from all $D \in \mathcal{D}(c)$. One expects to find that
\[
	M(c) \gg \abs{\mathcal{D}(c)},
\]
where $M(c)$ accounts for the multiplicity of equations arising from the same $D$. For example, the $abc$ equations for $D=p_i p_j$ can take the forms
\[
	F(p_i, p_j) = c, \quad F(p_j, p_i) = c, \quad F(1, p_i p_j) = c.
\]
For $D \in \mathcal{D}(c)$, let $\phi_{\omega}(D)$ be the number of all $D \in D_{\omega}$, where $D_{\omega}$ is as defined in \eqref{eq: 1.1}. Also let $M_{\omega}(c)$ denote the number of putative equations that can be formed with all the prime factors in each $D \in D_{\omega}$. Observe that
\begin{equation} \label{eq: 4.10}
	M_{0}(c) = 0, \quad M_{1}(c) = 1, \quad M_{2}(c) = 3, \quad M_{3}(c) = 7, \quad M_{4}(c) = 15, \quad \dots,
\end{equation}
where $M_{0}(c) = 0$ since no $abc$ equation can be formed when $\mathcal{D}(c) = \emptyset$, which is the case when $Q(c) = 1$. In fact, it is not hard to see that
\begin{equation} \label{eq: 4.11}
	M_{\omega}(c) = \sum_{r=0}^{\omega-1} C(\omega, r),
\end{equation}
where $C(n, k)$ is the binomial coefficient, with $C(n, k) =0$ for $k>n$. We have that
\begin{equation} \label{eq: 4.12}
	M(c) = \sum_{\omega = 1}^{\gamma} M_{\omega}(c)  \phi_{\omega}(D) = \sum_{\omega = 1}^{\gamma}\sum_{r=0}^{\omega-1} C(\omega, r) \phi_{\omega}(D),
\end{equation}
where
\[
	\gamma = \max\{\omega(D) : D \in \mathcal{D}(c)\}.
\]
This suggests the increasing complexity of finding $abc$ hits when $Q(c)$ gets larger in $c \in \mathbb{N}_{>2}$. We now state some fundamental facts to reinforce our argument.

\begin{lemma} \label{lem1}
	For $c \in \mathbb{N}_{>2}$ and $s \in \mathbb{S}$, if $c \in \overline{Q}$ and $\gcd(s, c)=1$, then $sc \in \overline{Q}$.
\end{lemma}

\begin{proof}
	Write $c=d_3Q(c)$, such that $c' = s d_3 Q(c)$ where $d'_3 = sd_3$ and the claim of the lemma thus immediately follows.
\end{proof}

\begin{lemma} \label{lem2}
	 For $c \in \overline{Q}$ and $s \in \mathbb{S}$, one has $s \overline{Q} \in F_{q}$ if $\gcd(s, c) = 1$ and $s \overline{Q} \in F_{sq}$ if $\gcd(s, c) > 1$.
\end{lemma}

\begin{proof}
	Lemma \ref{lem1} easily implies the case when $\gcd(s, c) = 1$. If $\gcd(s, c) > 1$, then $s | c$ for any $c \in \overline{Q}$. Since $s$ is squarefree, it follows that $s | d_3$. Let $\overline{Q'} = s\overline{Q}$ such that for any $c' \in \overline{Q'}$, we have that $c' = d_3 Q(c')$ where $Q(c') = sQ(c)$ with $\gcd(s, Q(c)) = 1$. Write $R(Q(c'))=q'$ and $R(Q(c))=q$. Then the latter claim holds since $q' = sq$, where $\gcd(s, q) = 1$.
\end{proof}

For some $D \in \mathcal{D}(c)$, we have shown that infinitely many hits result from $F(d_1, d_2)=c$ for $c \in F_{sq}$ or $E(d_3, d_2) = b$ for $c \in F_q$. Such hits are generated by an algorithm involving a recursion from a base hit with $F(d_1, d_2) \equiv 0 \pmod{z}$ or $E(d_3, d_1) \equiv 0 \pmod{y}$, respectively. For any base hit with $c_1 \in F_q$, define
\begin{equation} \label{eq: 4.13}
	\Phi(c_1, d_1, d_2) = \{c_i\}_{i=1}^{\infty} = \{c_1, c_2, c_3, \dots\}, \quad c_i < c_{i+1}, \quad \forall \; i \in \mathbb{N},
\end{equation}
as the sequence of hits generated by $F(d_1, d_2) = c$, whereas
\begin{equation} \label{eq: 4.14}
	\Psi(c_1, d_1, d_3) = \{c_j\}_{j=1}^{\infty} = \{c_1, c_2, c_3, \dots\}, \quad c_j < c_{j+1}, \quad \forall \; j \in \mathbb{N},
\end{equation}
by $E(d_3, d_1) = b$. Clearly, $\Phi(c_1, d_1, d_2)$ and $\Psi(c_1, d_1, d_3)$ are \textit{monotone increasing} in $F_{sq}$ and $F_q$, respectively. Note that $d_3$ varies but $d_1$ and $d_2$ are fixed in $\Phi(c_1, d_1, d_2)$, while $d_2$ varies but $d_1$ and $d_3$ are fixed in $\Psi(c_1, d_1, d_3)$. The $abc$ conjecture implies that the variations in $d_3$ and $d_2$ are not necessarily monotone in $\Phi(c_1, d_1, d_2)$ and $\Psi(c_1, d_1, d_3)$, respectively.  We record the fact that

\begin{lemma} \label{lem3}
	 For any base hit with $c_1 \in F_q$, there always exist $\Psi(c_1, d_1, d_2) \subset F_q$ and $\Phi(c_1, d_1, d_3) \subset F_{sq}$ for $s \geq 1$.
\end{lemma}

\begin{proof}
	Lemma \ref{lem3} generalizes Proposition \ref{prop5} and both claims are well exemplified here and there.
\end{proof}

To illustrate the case for $\Phi(c_1, d_1, d_2) \subset F_{sq}$, one can verify that for $c \in \overline{11}$, the $abc$ equation $F(1, 6) = 121K$ for $D=6$ generates infinitely many hits for $c \in F_{sq}$ of the form $1 + 2^{e_2}3^{e_3} = 121K$, where $e_2 = \mu_2 + k_2(t_2-1)$ and $e_3 = \mu_3 + k_3(t_3-1)$ with $\mu_2=55, t_2=110$ and $\mu_3=5, t_3=5$ for $z=11^2$, from which infinitely many hits exist for $c \in F_{sq}$ with $s=5$ and $q=11$ when $e_3 =  5 + 20k'_3$ for any $e_2$ in its specified form. For example, $1+2^{55}3^{5+5(k_3-1)} = c$ yields infinitely many hits for $c = 11^{n}K$ with $n>1$ and $K$ squarefree when $k_3 \not \equiv 1 \pmod{4}$, but $c = 5^{m}11^{n}K'$ with $m, n >1$ and $K'$ squarefree when $k_3 \equiv 1 \pmod{4}$, where $k_3$ is the number of cycles in $T_{121}(3)$.

We can write every $abc$ equation $a+b=c$ in its \textit{standard form} as
\begin{equation} \label{eq: 4.15}
	d_1Q(a) + d_2Q(b) = d_3Q(c),
\end{equation}
where
\begin{equation} \label{eq: 4.16}
	Q(a) = \prod_{p_{i}^{e_i} | a} p_{i}^{e_{i}-1}, \qquad Q(b) = \prod_{p_{j}^{e_j} | b} p_{j}^{e_{j}-1}, \qquad Q(c) = \prod_{p_{k}^{e_k} | c} p_{k}^{e_{k}-1},
\end{equation}
with $\gcd\big(Q(a), Q(b), Q(c)\big)=1$ and $\gcd(d_1, d_2, d_3)=1$, e.g. the hits $(1, 2^5 \cdot 7, 3^2 \cdot 5^2)$, $(1, 3^5 \cdot 5, 2^6 \cdot 19)$ and $(3^3 \cdot 17, 5^5, 2^9 \cdot 7)$ can be written in standard forms as
\[
	1 + 14 \cdot 2^4 =  15 \cdot 15, \qquad 1 + 15 \cdot 3^4 = 38 \cdot 2^5, \qquad 51 \cdot 3^2 + 5 \cdot 5^4 = 14 \cdot 2^8,
\]
respectively. In standard form, one can easily check if the hit satisfies the necessary condition that $d_1 d_2 d_3 < c$.

For $abc$ equations in standard form, $\Phi(c_1, d_1, d_2)$ and $\Psi(c_1, d_1, d_3)$ are the solution sets when $(d_1, d_2)$ or $(d_1, d_3)$ are fixed, respectively.. If $(a, b, c)$ is a base hit, where $d_2=w\mathcal{K}$ and $d_3=qK$, with $\mathcal{K}, K$ both squarefree, then there exists a hit $(a', b', c')$, where $c'>c$, $d'_2=w'\mathcal{K'}$ and $d'_3=q'K'$, with $\mathcal{K'}$ and $K'$ also both squarefree and $w, w'$ and $q, q'$ are not necessarily equal. For any $c' > c$, we say that $b$ or $c$ is \textit{lifted} if $Q(b')>Q(b)$ or $Q(c') > Q(c)$, respectively. Hence, we claim that

\begin{lemma} \label{lem4}
 	From a base hit with $c \in F_q$, lifting in $b$ or $c$ results in hits for $c' > c$, where $c' \in F_{sq}$ for $s \in \mathbb{S}$ and a fixed $q \in \mathbb{S}_{>1}$.
\end{lemma}

\begin{proof}
	Proposition \ref{prop3} or the well-known Hensel's lemma allows for any such tendency for lifting. Without loss of generality, assume that lifting occurs on a single prime factor $p_j$ in $b$ or $p_k$ in $c$. The lifting of  $p_j | b$ always leads to $c' \in F_q$. In the case of $c$, we have that $c' \in F_q$ if the lifting occurs on $p_k | q$, whereas $c' \in F_{sq}$ if on $p_k | K$. The latter case yields $Q(c')=p_kQ(c)$, such that $s = p_k$ where $p_k | K$ and $p_k \nmid q$, and this completes the proof.
\end{proof}

So far, we have not accounted for hits arising from base hits with both $\mathcal{K}>1$ and $K>1$. While our approach cannot deal with simultaneous solutions, i.e. where $\mathcal{K'}$ and $K'$ are both known, in such cases, it applies well if $d_1$ and $d_2$ are fixed to solve for $c$ or $d_1$ and $d_3$ are fixed to solve for $b$ by similar recursion with $F(d_1, d_2)=c$ or $E(d_3, d_1)=b$ to yield $\Phi(c_1, d_1, d_2)$ and $\Psi(c_1, d_1, d_3)$, respectively. Observe that $\Phi(c_1, d_1, d_2)$ and $\Psi(c_1, d_1, d_3)$ are sparse sequences, with densities as
\begin{equation} \label{eq: 4.17}
\begin{aligned}
	d_{\Phi(c_1, d_1, d_2)} & = \dfrac{\abs{\Phi(c_1, d_1, d_2) \cap [1, 2, 3, \dots, N]}}{N}, \\ 
	d_{\Psi(c_1, d_1, d_3)} & = \dfrac{\abs{\Psi(c_1, d_1, d_3) \cap [1, 2, 3, \dots, N]}}{N},
\end{aligned}
\end{equation}
for $N \in \mathbb{N}$. We obviously find that
\begin{equation} \label{eq: 4.18}
	\lim_{N \to \infty} d_{\Phi(c_1, d_1, d_2)} = 0, \qquad \lim_{N \to \infty} d_{\Psi(c_1, d_1, d_3)} = 0.
\end{equation}
The integers in $\Phi(c_1, d_1, d_2)$ and $\Psi(c_1, d_1, d_3)$ are not necessarily unique due to the likely occurrence of super exceptions. In addition, some base hits or integers close to the base hits have relatively high qualities in $\Phi(c_1, d_1, d_2)$ and $\Psi(c_1, d_1, d_3)$, which is not surprising given the vanishing density of perfect squares in $\mathbb{N}$. In fact, the Reyssat hit $(2, 3^{10} \cdot 109, 23^{5})$, which has the highest quality among the hits known so far, is a base hit. An infinitude of hits with lesser qualities can thus be generated from the Reyssat hit. To show this, take $\mu_2=1$, $\mu_3=10$ and $t_2 = t_3=11 \cdot 23^3$ for modulus $z=23^4$ to obtain
\begin{equation} \label{eq: 4.19}
	2^{ 1+133837(k_2-1)} + 109 \cdot 3^{10+133837(k_3-1)} = 23^{4+e_k}K, \quad e_k \geq 0,
\end{equation}
which yields infinitely many hits for any $k_2, k_3 \in \mathbb{N}$, with $K \not \equiv 0 \pmod{23}$. Unless proven otherwise, $K$ is not necessarily squarefree in view of the potential lifting of any prime factor coprime to $z$ in $c$ by Lemma \ref{lem4}. On the other hand, with $t_2=t_{23}=18$ for modulus $y=3^3$, we also find an infinitude of hits with $E(d_3, d_1) = c-a = b$. A simplified $abc$ equation in the form
\begin{equation} \label{eq: 4.20}
	b = 23^{5n} - 2^n, \quad \forall \; n \in \mathbb{N}
\end{equation}
certainly, but not exclusively, generates infinitely many hits from the Reyssat hit. Notice that the squarefree factor $\mathcal{K}$ in $b$ drastically increases with every unit increase in $n$, thus constraining the growth in $\lambda(a, b, c)$ by virtue of \eqref{eq: 4.4}. With \eqref{eq: 4.4}, \eqref{eq: 4.5}, \eqref{eq: 4.7}, \eqref{eq: 4.13} and \eqref{eq: 4.14}, the $abc$ conjecture thus follows naturally from any rigorous proof showing that
\begin{equation} \label{eq: 4.21}
\begin{aligned}
	& K_{i+1} > K_{i}, \quad c_i, c_{i+1} \in \Phi(c_1, d_1, d_2), \\ 
	& \mathcal{K}_{j+1} > \mathcal{K}_{j}, \quad c_j, c_{j+1} \in \Psi(c_1, d_1, d_3),
\end{aligned}
\end{equation}
with only finitely many exceptions. This implies that $R(c_{i+1}) > R(c_i)$ and $R(b_{j+1}) > R(b_j)$ infinitely often for $i, j \in \mathbb{N}$ in $\Phi(c_1, d_1, d_2)$ and $\Psi(c_1, d_1, d_3)$, respectively, such that \eqref{eq: 4.21} tends to coincide with the conjectured limit in \eqref{eq: 4.3}. In line with Lemma \ref{lem4}, notice that $E(d_3, d_1)=b$ also results in the lifting of $b$, where $Q(b')>Q(b)$ for $c'>c$, to generate the hits in $\Psi(c_1, d_1, d_3)$, as exemplified above by $E(2,5)=b$ with $(5, 3^3, 2^5)$ as the base hit. Such a case demonstrates that $\mathcal{K}_{j+1} > \mathcal{K}_{j}$ for most $c_j, c_{j+1} \in \Psi(c_1, d_1, d_3)$ since the incremental effect of any lifting on  primes in $d_2$ is expected to be less than the growth of $c$ in the sequence, except for only finitely many exceptions. Note, however, that depending on the choice of modulus $y$ from a base hit, the resulting solutions to $E(d_3, d_1)=b$ may not entirely constitute hits, as shown for $E(2, 5)=27\mathcal{K}$ from the base hit $(5, 3^3, 2^5)$.

Every integer $N$ can be expressed in the form
\begin{equation} \label{eq: 4.22}
	N = sf(m), \qquad m,s \in \mathbb{S}, \qquad \gcd(m, s)=1,
\end{equation}
where $s$ is the squarefree part and $f(m)$ is the squareful part. If $N$ is squarefree, then $m$ is trivial $(m=1)$. Hence, every $abc$ equation can be written as
\begin{equation} \label{eq: 4.23}
	h_1 f(x_1) + h_2 f(x_2) = h_3 f(x_3),
\end{equation}
where $x_i$ and $h_i$ $(i = 1, 2, 3)$ are all squarefree and pairwise coprime, except when $a=h_1f(x_1)=1$ or at least two terms in $(a, b, c)$ are powerful numbers. It follows from \eqref{eq: 4.15} and \eqref{eq: 4.23} that 
\begin{equation} \label{eq: 4.24}
	Q(a) = \dfrac{f(x_1)}{x_1}, \qquad Q(b) = \dfrac{f(x_2)}{x_2}, \qquad Q(c) = \dfrac {f(x_3)}{x_3},
\end{equation}
such that
\begin{equation} \label{eq: 4.25}
	d_1 = h_1 x_1, \qquad d_2 = h_2 x_2, \qquad d_3 = h_3 x_3,
\end{equation}
where $x_2 = w$ and $x_3=q$ for $w$ and $q$ as above. Let $x$ be a squarefree composite, with prime decomposition as
\begin{equation} \label{eq: 4.26}
	x = p_1 p_2 \cdots p_k,
\end{equation}
where $p_i$ $(i=1, 2, \dots, k)$ denotes all $p_i | x$, with $\omega(x)=k$. Suppose that
\begin{equation} \label{eq: 4.27}
	f(x) = p_1^{e_1} p_2^{e_2} \cdots p_k^{e_k} = n, \qquad e(n) = (e_1, e_2, \dots, e_k),
\end{equation}
then for some modulus $z$, as specified above, we have that
\begin{equation} \label{eq: 4.28}
	r_{e(n)} \equiv \prod_{p_i | x} r_{e_i} \pmod{z}, \qquad r_{e_i} \in T_{z}(p_i),
\end{equation}
where $T_{z}(p_i)$ is as defined in \eqref{eq: 3.13}. For $h \in \mathbb{S}_{>1}$, we obtain
\begin{equation} \label{eq: 4.29}
	hn \equiv hr_{e(n)} \equiv r'_{e'(n)} \pmod{z}, \qquad e'(n) = \big(1, e(n)\big),
\end{equation}
where each prime dividing $h$ has exponent $1$. We also note some basic facts relevant to the hit-finding algorithm as outlined above.

\begin{lemma} \label{lem5}
 	For $D \in \mathcal{D}(c)$, if $H(c)>0$ for $D=s_1 s_2$ where $s_1, s_2 \in \mathbb{S}_{>1}$, then both $F(s_1, s_2) = c$ and $F(s_2, s_1)=c'$ generate infinitely many hits for $c, c' \in F_{sq}$, with $f(s_1) < f(s_2)$ and $f(s_2) < f(s_1)$, respectively.
\end{lemma}

\begin{proof}
	The claim is trivial given the periodicities of $T_{z}(p_i)$ and $T_{z}(p_j)$, where $z$ is as defined from $c$ of the base hit, for every prime $p_i | s_1$ and every prime $p_j | s_2$, from which one can establish infinite solutions $U_{z}(s_1; s_2)$ to $F(s_1, s_2) = c$ with $f(s_1)<f(s_2)$ and $U_{z}(s_2; s_1)$ to $F(s_2, s_1) = c'$ with $f(s_2)<f(s_1)$.
\end{proof}

Lemma \ref{lem5} is exemplified above by $F(2, 3) = 49K$ and $F(3, 2) = 49K$. The same lemma provides the basis for establishing the next two larger hits from the Reyssat hit as 
\[
	(109 \cdot 3^{10}, 2^{133838}, 23^{4}K') \quad \textnormal{and} \quad (2, 109 \cdot 3^{133847}, 23^{4}K), 
\]
where $K'<K$ are extremely large squarefree integers with $K', K \not \equiv 0 \pmod{23}$.

\begin{lemma} \label{lem6}
 	From the same base hit, both $E(d_3, d_1) = b$ and $E(d_3, d_2)=a$ generate infinitely many hits for $c \in F_{q}$, where $c=f(d_3)$ and $1<d_1=d_2$ between equations.
\end{lemma}

\begin{proof}
	If there exist solutions $U_{y}(d_3; d_1)$ to $E(d_3, d_1) = b$ where $y$ is as defined from a given base hit, then one naturally finds solutions $U_{y}(d_3; d_2)$ to $E(d_3, d_2)=a$ by suitable choices of the power vectors for $s \in \mathbb{S}_{>1}$, where $d_1=s$ in $E(d_3, d_1)=b$ and $d_2=s$ in $E(d_3, d_2)=a$ such that $1<d_1=d_2$ between equations. Both solution sets constitute hits for $c \in F_q$, where $c = f(d_3)$ and $d_1 d_2 < Q(c)$.
\end{proof}

Lemma \ref{lem6} implies the infinitudes of hits that can be generated by, say,  $E(2,5)=b$ and $E(2,5)=a$ from the base hit $(5, 3^3, 2^5)$, such as
\[
	(5^3, 3^4 \cdot 13 \cdot 31, 2^{15}) \quad \textnormal{and} \quad (3^3 \cdot 37 \cdot 53, 5^7, 2^{17}).
\]
Hence, Lemmas \ref{lem5} and \ref{lem6} demonstrate the additive commutativity of $abc$ equations, provided it is clear that $a=f(d_1) < b=f(d_2)$. Both lemmas also have important consequences. For instance, for $D<Q(c)$ let $D=s_1s_2$, where $s_1 \in \mathbb{S}$ and $s_2 \in \mathbb{S}_{>1}$. If $s_1=1$ and $\Phi(c_1, s_1, s_2) \in F_{sq}$, then it is trivial that $\Phi(c_1, s_2, s_1) = \emptyset$. If $1<f(s_1)<f(s_2)$, then $\Phi(c_1, s_1, s_2) \cup \Phi(c_1, s_2, s_1) \in F_{sq}$. Similarly, if $s_1=1$ and $\Psi(c_1, s_1, d_3) \in F_{q}$, then $\Phi(c_1, s_2, d_3) = \emptyset$, where $\Phi(c_1, s_2, d_3)$ denotes the sequence of $c$ for which $c - f(s_2) = a$. Hence, any inferences on the same theme as \eqref{eq: 4.7} similarly apply to $\Phi(c_1, s_2, d_3)$ for $a>1$. If $1<f(s_1)<f(s_2)$, then one also has $\Psi(c_1, s_1, d_3) \cup \Psi(c_1, s_2, d_3) \in F_{q}$. We can generalize these patterns by

\begin{lemma} \label{lem7}
 	For $D<Q(c)$ depending on $c \in \mathbb{N}_{>2}$, let $D=s_1 s_2$ where $f(s_1) < f(s_2)$ for $s_1 \in \mathbb{S}$, $s_2 \in \mathbb{S}_{>1}$. Given any base hit $(a_1, b_1, c_1)$, we have that
\[
	\Phi(c_1, D) = \bigcup_{\substack{(s_1, s_2) \\ s_1s_2=D}} \Phi(c_1, s_1, s_2), \qquad \Psi(c_1, D) = \bigcup_{\substack{(s_1, s_2) \\ s_1s_2=D}} \Psi(c_1, s_i, d_3).
\]
such that
\[
	\Phi(c_1, \mathcal{D}(c)) = \bigcup_{D \in \mathcal{D}(c)} \Phi(c_1, D), \qquad \Psi(c_1, \mathcal{D}(c)) = \bigcup_{D \in \mathcal{D}(c)} \Psi(c_1, D).
\]
\end{lemma}

\begin{proof}
	This lemma naturally follows from the above narrative, where the specified unions account for the multiplicity of $abc$ equations, denoted by $M(c)$ above, for some $D \in \mathcal{D}(c)$ depending on $c \in \mathbb{N}_{>2}$ in $\Phi$ and $\Psi$.
\end{proof}

It is clear from above that $D$ has a fixed range in $\Phi(c_1, \mathcal{D}(c))$, but $D$ is unbounded in $\Psi(c_1, \mathcal{D}(c))$. For distinction, let $\mathcal{D}_{\Phi}(c)$ and $\mathcal{D}_{\Psi}(c)$ be the admissible sets of $D$ in $\Phi$ and $\Psi$, respectively. Denote by $\Gamma(c_1)$ the set of all exceptional $c$ arising from any known base hit with $c_1$, such that

\begin{lemma} \label{lem8} 
\[
	\Gamma (c_1) = \Phi(c_1, \mathcal{D}_{\Phi}(c))  \cup \Psi(c_1, \mathcal{D}_{\Psi}(c)).
\]
\end{lemma}

\begin{proof}
	The equation immediately follows from Lemma \ref{lem7}, with a change in notation as defined above.
\end{proof}

Lemma \ref{lem8} accounts for any $c_1$ where $H(c_1) \geq 1$. The $abc$ conjecture thus follows naturally if \eqref{eq: 4.5} and \eqref{eq: 4.7} hold for all $c \in \Gamma (c_1)$ regardless of $c_1 \in \mathbb{N}_{>2}$. Hopefully, this provides useful insights to motivate simpler geometric constructions that could lead to a more rigorous proof of the conjecture.

\bibliographystyle{amsplain}

\begin{thebibliography}{12}

\bibitem{BM} B. Mazur, \textit{Questions about powers of numbers}. Notices of the American Mathematical Society 47(2), 195-202 (2000).
\bibitem{JO} J. Oesterl\`e, \textit{Nouvelles approches du theor\`em\'e de Fermat} (New approaches to Fermat's last theorem). S\'eminaire Bourbaki, 40eme Annee, Vol. 1987/88, Exp. No. 694, Ast\'erisque 161/162, 165-186 (1988).
\bibitem{DWM} D.W. Masser, \textit{Note on a conjecture of Szpiro}. S\'eminaire sur les pinceaux de courbes elliptiques, Paris, France 1988, Ast\'erisque 183, 19-23 (1990).
\bibitem{MW} M. Waldschmidt, \textit{Lectures on the $abc$ conjecture and some of its consequences}. In: P. Cartier, A.D.R. Choudary, M. Waldschmidt (eds.), Springer Proceedings in Mathematics and Statistics 98, 211-230 (2015).
\bibitem{MM} G. Martin, W. Miao, \textit{$abc$ triples}. Functiones et Approximatio 55.2, 145-176 (2016).
\bibitem{NE} N.D. Elkies, \textit{The ABC's of number theory}. The Harvard College Mathematics Review 1(1), 57-76 (2007).
\bibitem{FM} F. Mertens, \textit{Uber einige asymptotische Gesetze der Zahlentheorie} (On several asymptotic laws in number theory). Crelle's Journal 77, 289-338 (1874).
\bibitem{AW} A. Walfisz, \textit{Weylsche Exponentialsummen in der neueren Zahlentheorie} (Weyl exponential sums in the newer number theory). VEB Deutscher Verlag der Wissenschaften, Berlin (1963).
\bibitem{MP} G. Mincu, L. Panaitopol, \textit{On some properties of squarefree and squareful numbers}. Bulletin Math\'ematique de la Soci\'et\'e des Sciences Math\'ematiques de Roumanie 48 (97), No. 1, 63-88 (2006).
\bibitem{LZ} S.K. Lando, A.K. Zvonkin, \textit{Graphs on Surfaces and Their Applications}. Encyclopedia of Mathematics Sciences: Lower-Dimensional Topology II 141, Springer-Verlag (2004).
\bibitem{JB} J. Browkin, \textit{The $abc$-conjecture}. In: R.P. Bambah, V.C. Dumir, and R.J. Hans-Gill (eds.), Number Theory, Trends in Mathematics, Birkh\"{a}user, Basel, 75-100 (2000).
\bibitem{BB} J. Browkin, J. Brzezi\'{n}ski, \textit{Some remarks on the $abc$-conjecture}. Mathematics of Computation 62(206) 931-939 (1994).

\end{thebibliography}

\end{document}